\documentclass[conference,letterpaper]{IEEEtran}

\usepackage{amssymb,amsmath,amsthm,mathrsfs,color}
\usepackage{subfigure}
\usepackage{graphicx}
\usepackage{graphics}
\usepackage{enumitem}
\usepackage{multirow}
\usepackage{dsfont}
\usepackage{enumitem}
\usepackage{ulem}
\usepackage{ascmac}
\usepackage{bm}
\usepackage{tikz}

\newcommand{\Cov}[0]{\mathsf{Cov}}
\newcommand{\Var}[0]{\mathsf{Var}}

\newcommand{\E}[0]{\mathbb{E}}

\newcommand{\R}[0]{\mathbb{R}}

\newcommand{\Prob}[0]{\mathbb{P}}

\renewcommand{\tilde}{\widetilde}

\newcommand{\vertiii}[1]{{\left\vert\kern-0.25ex\left\vert\kern-0.25ex\left\vert #1 
    \right\vert\kern-0.25ex\right\vert\kern-0.25ex\right\vert}}
\newcommand\numberthis{\addtocounter{equation}{1}\tag{\theequation}}

\newcommand{\cB}{\mathcal{B}}

\newcommand{\cN}{\mathcal{N}}

\newcommand{\cP}{\mathcal{P}}
\newcommand{\cQ}{\mathcal{Q}}

\newcommand{\EE}{\mathbb{E}}

\newcommand{\PP}{\mathbb{P}}

\newcommand{\RR}{\mathbb{R}}

\newcommand*{\dd}{\, \mathsf{d}}
\newcommand{\vasti}{\bBigg@{3.5 }}
\newcommand{\vast}{\bBigg@{4}}
\newcommand{\Vast}{\bBigg@{5}}
\newcommand{\Vastt}{\bBigg@{7}}
\makeatother
\newcommand{\be}{\begin{equation}}
\newcommand{\ee}{\end{equation}}
\newcommand{\ba}{\begin{align}}
\newcommand{\ea}{\end{align}}
\newcommand{\baa}{\begin{align*}}
\newcommand{\eaa}{\end{align*}}

\newcommand*{\gauss}{\varphi_\sigma}

\newcommand*{\Gauss}{\mathcal{N}_\sigma}

\newcommand{\tv}{\delta_{\mathsf{TV}}}
\newcommand{\gtv}{\delta_{\mathsf{TV}}^{(\sigma)}}
\newcommand{\Chi}{\chi^2}
\newcommand{\gchi}{\chi^2_{\sigma}}

\newtheorem{theorem}{Theorem}
\newtheorem{lemma}{Lemma}
\newtheorem{proposition}{Proposition}
\newtheorem{corollary}{Corollary}

\newtheorem{definition}{Definition}

\begin{document}

\title{Limit Distributions for Smooth Total Variation and $\chi^2$-Divergence in High Dimensions} 
\author{
  \IEEEauthorblockN{Ziv Goldfeld}
  \IEEEauthorblockA{Cornell University\\
                    goldfeld@cornell.edu}
  \and
  \IEEEauthorblockN{Kengo Kato}
  \IEEEauthorblockA{Cornell University\\
                    kk976@cornell.edu}
\IEEEoverridecommandlockouts
\IEEEcompsocitemizethanks{
\IEEEcompsocthanksitem
}}
\maketitle

\begin{abstract}
    Statistical divergences are ubiquitous in machine learning as tools for measuring discrepancy between probability distributions. As these applications inherently rely on approximating distributions from samples, we consider empirical approximation under two popular $f$-divergences: the total variation (TV) distance and the $\Chi$-divergence. To circumvent the sensitivity of these divergences to support mismatch, the framework of Gaussian smoothing is adopted. We study the limit distributions of $\sqrt{n}\tv(P_n\ast\Gauss,P\ast\Gauss)$ and $n\Chi(P_n\ast\Gauss\|P\ast\Gauss)$, where $P_n$ is the empirical measure based on $n$ independently and identically distributed (i.i.d.) observations from $P$, $\Gauss:=\mathcal{N}(0,\sigma^2\mathrm{I}_d)$, and $\ast$ stands for convolution. In arbitrary dimension, the limit distributions are characterized in terms of Gaussian process on $\R^d$ with covariance operator that depends on $P$ and the isotropic Gaussian density of parameter $\sigma$. This, in turn, implies optimality of the $n^{-1/2}$ expected value convergence rates recently derived for $\tv(P_n\ast\Gauss,P\ast\Gauss)$ and $\Chi(P_n\ast\Gauss\|P\ast\Gauss)$. These strong statistical guarantees promote empirical approximation under Gaussian smoothing as a potent framework for learning and inference based on high-dimensional data.
\end{abstract}

\maketitle


\section{Introduction}





Statistical divergences are central to many fields, such as machine learning (ML), information theory and statistics. They quantify discrepancy between probability measures, which is central to those fields. Two fundamental statistical divergences are the total variation (TV) distance and the $\Chi$-divergence, on which we focus herein. TV and $\Chi$ fall under the broader framework of $f$-divergences \cite{csiszar2004information}. As such, they possess an array of important properties (data processing inequality, variational representation, etc.), making them appealing for analyzing and designing inference systems --- see, e.g., \cite{nowozin2016f,uehara2016generative} for  recent applications of $f$-divergences to generative modeling. Focusing on statistical applications, where only samples of the underlying distributions are available, naturally leads to the question of empirical approximation under TV and $\Chi$.

Suppose $P_n$ is the empirical measure induced by $n$ independently and identically distributed (i.i.d.) observations from a $d$-dimensional distribution $P$. We would like to consider the rate at which $P_n$ approaches $P$ under TV and $\Chi$. However, in general, $P_n$ does not converge to $P$ under the TV topology (e.g., if $P$ is absolutely continuous with respect to (w.r.t.) Lebesgue) \cite{devroye1990no}, while $\Chi(P\|Q)=\infty$ whenever $P$ is not absolutely continuous w.r.t. $Q$. To obtain a well-posed empirical approximation setup, we adopt the Gaussian smoothing framework of \cite{Goldfeld2019convergence} (see also \cite{Goldfeld2020GOT,goldfeld2020limit}). Accordingly, we define $\gtv(P_n,P):=\tv(P_n\ast\Gauss,P\ast\Gauss)$ and $\gchi:=\Chi(P_n\ast\Gauss\|P\ast\Gauss)$, where $P_n$ and $P$ are both convolved with an isotropic Gaussian measure $\Gauss:=\cN(0,\sigma^2\mathrm{I}_d)$. This alleviates the aforementioned pathologies since now both measures have densities supported on the entire space. Furthermore, \cite{Goldfeld2019convergence} showed that $\EE[\gtv(P_n,P)] = O(n^{-1/2})$ and $\EE[\gchi(P_n\|P)] = O(n^{-1})$ for subgaussian $P$ in any dimension $d$.\footnote{The full results have a prefactor of $c^d$ that quantifies the dependence on dimension. Still, $n$ and $d$ are decoupled and treating $d$ as fixed results in a convergence rate independent of it.} Our objective is to upgrade these results by characterizing the limit distributions of their (normalized) versions and establish optimality of the above rates. 

Building on empirical process theory and probability theory in Banach spaces\footnote{The reader is referred to, e.g., \cite{LeTa1991,VaWe1996,GiNi2016} as useful references.}, we characterize the limit distributions of $\sqrt{n}\gtv(P_n,P)$ and $n\gchi(P_n\|P)$ as $n\to\infty$. Both limits are given in terms of an integral operator of a centered Gaussian process $B_P^{(\sigma)}:=\big(B_P^{(\sigma)}(x)\big)_{x\in\R^d}$. The covariance function of $B_P^{(\sigma)}$ depends on the data distribution $P$ and the noise parameter $\sigma$. Gaussian-smoothing is crucial here since it allows reasoning about TV and $\Chi$ in terms of $L^p$ norms. With that perspective, the limit distribution results follow from the central limit theorem (CLT) in $L^1$ for TV and $L^2$ for $\Chi$. A direct consequence of the limit distribution results is the optimality of the expected value convergence rates derived in \cite{Goldfeld2019convergence}, which was not established therein. A concentration inequality for $\gtv$ via McDiarmid's inequality is also derived. Our results hold under milder assumptions on $P$ than in \cite{Goldfeld2019convergence}.

Most related to this work is \cite{goldfeld2020limit}, where the limit distribution of smooth empirical 1-Wasserstein distance was derived for all dimensions (the proof techniques employed therein are quite different from those in this work). A corresponding result for unsmoothed 1-Wasserstein is only known in the 1-dimensional case \cite{del1999central}. Limit distributions under TV, $\Chi$ and other $f$-divergences, were also studied before but under a framework different than ours. For example, \cite{bobkov2014berry} and \cite{bobkov2019renyi} consider the TV distance and $\Chi$-divergence between the distribution of a normalized sum of independent random variables and a Gaussian. They find conditions under which convergence holds and characterise the corresponding rates. Another related work is \cite{courtade2020bounds} where bounds on the Poincar{\'e} constant of convolved measures were derived, with application to limit distribution questions under the same framework as \cite{bobkov2014berry,bobkov2019renyi} (see also \cite{bobkov2019entropic}). These results significantly differ from those presented herein, as we focus on (smooth) empirical TV and $\Chi$, i.e., $\sqrt{n}\gtv(P_n,P)$ and $n\gchi(P_n\|P)$, as the random variable sequence of interest (as opposed to distance measure between the distribution of a normalized sum and its Gaussian limit). 

From a broader perspective, our results suggest that  smooth statistical divergences may be particularly well-suited for high-dimensional learning and inference tasks. Statistical divergences can be used to formulate a variety of ML tasks, from generative modeling to outlier detection to ensemble methods. Such formulations land themselves well for a theoretical analysis, but, their usefulness diminishes as the data dimension grows larger. This is because statistical divergences suffer from the curse of dimensionality (CoD) in empirical approximation, with error convergence rates effectively decaying as $n^{-1/d}$, where $n$ is the sample size and $d$ is the data dimension (see, e.g.,\cite{FournierGuillin2015}, \cite{nguyen2010estimating} and \cite{sriperumbudur2009integral} for CoD results for Wasserstein distances \cite{villani2008optimal}, $f$-divergences and integral probability metrics \cite{muller1997integral}, respectively). On the other hand, when smoothing is introduced, empirical convergence of TV and $\Chi$ accelerates to dimension-free rates, posing then them as favorable figures of merit.




\textbf{Notation:} Let $\| \cdot \|$ denote the Euclidean norm, and $x \cdot y$, for $x,y \in \R^{d}$, denote their inner product. Let $\cP(\RR^d)$ denote the class of Borel probability measures on $\RR^d$, and $\cB(\RR^d)$ denote the Borel $\sigma$-algebra on $\R^{d}$. The isotropic Gaussian measure on $\RR^d$ with parameter $\sigma > 0$ is denoted as  $\Gauss:=\cN(0,\sigma^2\mathrm{I}_d)$ and its probability density function is denoted as $\gauss$, i.e., $\varphi_{\sigma}(x) = (2\pi \sigma^{2})^{-d/2} e^{-\|x\|^{2}/(2\sigma^{2})}$, for $x \in \R^{d}$. Given $P,Q\in\cP(\RR^d)$, their convolution $P\ast Q\in\cP(\RR^d)$ is defined as $(P\ast Q)(\mathcal{A})=\int\int\mathds{1}_{\mathcal{A}}(x+y)\dd P(x)\dd Q(y)$ for $\mathcal{A} \in \cB(\R^{d})$, where $\mathds{1}_\mathcal{A}$ is the indicator function of $\mathcal{A}$. The convolution between $P \in \mathcal{P}(\R^{d})$ and a measurable function $f$ on $\R^{d}$ is defined as $P\ast f (x) = \int_{\R^{d}} f(x-y) \dd P(y)$, for $x \in \R^{d}$, whenever the latter integral is well-defined for all $x \in \R^{d}$ (cf. \cite{Folland1999} p. 271). 

Given $P \in \cP(\R^{d})$ and $X_1,\dots,X_n \sim P$ i.i.d., the empirical distribution is defined by $P_n:= n^{-1} \sum_{i=1}^n \delta_{X_i}$ where $\delta_x$ denotes the Dirac measure at $x$. A stochastic process $B=\big(\mspace{-1.5mu}B(t)\mspace{-1.5mu}\big)_{\mspace{-3mu}t\in T}$ is called Gaussian if for any finite subset $F$ of $T$, $(B(t))_{t \in F}$ is Gaussian. A version of $B$ is another stochastic process with the same finite dimensional distributions.


\section{Limit Distribution Results}
\label{sec: related metrics}

We derive limit distributions of smooth TV distance and $\Chi$-divergence between $P_n$ and $P$. These results rely substantially on CLTs in $L^p$ spaces, and for the reader's convenience, we summarize basic CLT results in $L^p$ spaces in Appendix \ref{APPEN:CLT in Lp}.


\subsection{Smooth Total Variation Distance}

Consider $\tv(P,Q):= \sup_{A\in\cB(\RR^d)}\big|P(A) - Q(A)\big|$ the TV distance between $P,Q\in\cP(\RR^d)$, and define its smooth version as $\gtv(P,Q):=\tv(P\ast\Gauss,Q\ast\Gauss)$. Let $B_P^{(\sigma)} = \big(B_{P}^{(\sigma)}(x)\big)_{x \in \R^{d}}$ denote a centered Gaussian process with covariance function
\begin{equation*}
\E\left[B_{P}^{(\sigma)}(x)B_{P}^{(\sigma)}(y)\right]=\Cov_{P} \big(\gauss(x-\cdot),\gauss(y-\cdot)\big),
\end{equation*}
for all $x,y\in\R^{d}$. The following theorem derives a limit distribution and moment bound for $\gtv(P_n,P)$ under a certain moment condition on~$P$. 

\begin{theorem}[Limit distribution for $\gtv$]
\label{thm: CLT TV}
If 
\begin{equation}
\label{eq: TV condition}
\int_{\R^{d}} \sqrt{\Var_{P}(\gauss(x-\cdot))}\dd x < \infty,
\end{equation}
then there exists a version of $B_{P}^{(\sigma)}$ that is an $L^{1}(\R^{d})$-valued random variable such that
\begin{subequations}
\begin{equation}
\sqrt{n}\gtv(P_n,P)\stackrel{d}{\rightarrow} \frac{1}{2} \int_{\R^{d}} \big| B_{P}^{(\sigma)}(x)\big|\dd x.     
\end{equation}
In addition, we have 
\begin{equation}
\sqrt{n}\E\Big[\gtv(P_n,P)\Big] \le \frac{1}{2} \int_{\R^{d}} \sqrt{\Var_{P}(\gauss(x-\cdot))}\dd x. 
\end{equation}
\end{subequations}
\end{theorem}


The following lemma derives a sufficient condition for Condition (\ref{eq: TV condition}) to hold. 

\begin{lemma}[Sufficient condition for \eqref{eq: TV condition}]
\label{lem: TV condition}
We have for $X \sim P$,
\begin{equation}
\begin{split}
&\int_{\R^{d}} \sqrt{\Var_{P}(\gauss(x-\cdot))}\dd x \\
&\quad  \le 8^{d/2} + \frac{2^{d/2+1}}{\sigma^{d}\Gamma(d/2)} \int_{0}^{\infty}t^{d-1}\sqrt{\Prob(\|X\| > t)} \dd t.
\end{split}
\label{eq: TV moment bound}
\end{equation}
Thus, if $P$ has finite $(2d + \epsilon)$ moments for some $\epsilon > 0$, i.e., $\E[\| X \|^{2d+\epsilon}] < \infty$, then Condition (\ref{eq: TV condition}) holds.  
\end{lemma}


Proposition 2 in \cite{Goldfeld2019convergence} derives a moment bound on $\gtv(P_nP)$ assuming  sub-Gaussian $P$. This condition is substantially relaxed in Theorem \ref{thm: CLT TV} above (in addition to deriving a limit distribution). In fact, Condition \eqref{eq: TV condition} is sharp for the first moment of $\gtv(P_n,P)$ to be of order $n^{-1/2}$. The following proposition shows that if \eqref{eq: TV condition} does not hold, then $\gtv(P_n,P)$ has a rate strictly slower than $n^{-1/2}$. We say that a sequence of random variables $Y_n$ is stochastically bounded (or tight) if for any $\epsilon > 0$, there exists a constant $M = M_{\epsilon}$ such that $\PP (|Y_{n}| > M) \le \epsilon$ for all $n$.

\begin{proposition}[Sharpness of Condition \eqref{eq: TV condition}]
\label{cor: TV condition}
If $\sqrt{n}\gtv(P_n,P)$ is stochastically bounded, then Condition \eqref{eq: TV condition} holds. 
\end{proposition}

The proof indeed shows that 
\begin{equation*}
\liminf_{n} \sqrt{n} \E\Big[\gtv(P_n,P)\Big] \ge \frac{1}{\sqrt{2\pi}} \int_{\R^{d}} \sqrt{\Var_{P}(\gauss(x-\cdot))}\dd x
\end{equation*}
\textit{without} assuming Condition \eqref{eq: TV condition}. The stochastic boundedness of $\sqrt{n}\gtv(P_n,P)$ implies the boundedness of the first moment by Hoffmann-J{\o}rgensen's inequality, so the conclusion of the corollary follows. See Section \ref{SUBSEC: TV condition proof} for details.

Finally, we state a concentration inequality for $\gtv(P_n,P)$. 

\begin{corollary}[Concentration inequality for $\gtv$]
\label{cor: concentration TV}
Under Condition \eqref{eq: TV condition}, we have 
\begin{equation}
\PP\bigg (\gtv(P_n,P) \ge \E\Big[\gtv(P_n,P)\Big] + t \bigg) \le e^{-nt^{2}/2},\quad\forall t > 0.
\end{equation}
\end{corollary}
This result follows from a simple application of McDiarmid's inequality (cf. \cite{Mc1989} or \cite{GiNi2016} Theorem 3.3.14), together with the fact that $\| \gauss(\cdot-X_i) \|_1 = \int_{\R^{d}} \varphi_{\sigma}(x-X_{i}) \dd x = 1$ (cf. the proof of Theorem \ref{thm: CLT TV}). We omit the details for brevity.


\subsection{Smooth $\Chi$-Divergence}

The $\Chi$-divergence between $P$ and $Q$ is $\chi^{2} (P \| Q) := \int \left(\frac{\dd P}{\dd Q} - 1\right)^{2}\dd Q$.
We have the following limit distribution result for its smooth empirical version $\gchi(P\|Q):=\Chi(P\ast\Gauss\|Q\ast\Gauss)$.
Recall that $P*\gauss(x) := \int_{\R^{d}} \gauss(x-y)\dd P(y)$.

\begin{theorem}[Limit distribution for $\gchi$]
\label{thm: CLT chi2}
If 
\begin{equation}
\label{eq: chi2 condition}
\int_{\R^{d}} \frac{\Var_{P}(\gauss(x-\cdot))}{P*\varphi (x)}\dd x < \infty,
\end{equation}
then there exists a version of $B_{P}^{(\sigma)}$ such that $B_{P}^{(\sigma)}/\sqrt{P*\gauss}$ is an $L^{2}(\R^{d})$-valued random variable and 
\begin{subequations}
\begin{equation}
n \gchi(P_n\| P) \stackrel{d}{\rightarrow} \int_{\R^{d}}  \frac{\big| B_{P}^{(\sigma)}(x)\big|^2}{P*\gauss (x)}\dd x.
\end{equation}
In addition, we have 
\begin{equation}
n \E\big[\gchi(P_n\|P)\big] = \int_{\R^{d}} \frac{\Var_{P}(\gauss(x-\cdot))}{P*\gauss (x)}\dd x. 
\end{equation}
\end{subequations}
\end{theorem}

Condition \eqref{eq: chi2 condition} holds if $P$ is $\beta$-sub-Gaussian with $\beta > \sigma/\sqrt{2}$. 

\begin{definition}[Sub-Gaussian distribution]\label{DEF:SG}
We call $P\in\cP(\R^d)$ $\beta$-sub-Gaussian, for $\beta >0$, if $X \sim P$ satisfies 
\begin{equation*}
\E\left [ \exp \big ( \alpha \cdot (X-\E[X]) \big) \right ]\le e^{\beta^{2}\| \alpha \|^{2}/2}, \quad \forall \alpha \in \R^{d}. 
\end{equation*}
\end{definition}

Let $Z\sim\cN (0,\mathrm{I}_{d})$. By Definition \ref{DEF:SG}, for any $\beta$-sub-Gaussian $X$ with mean zero and $0 \le \eta < 1/(2\beta^{2})$, we have
\begin{align*}
\E\big [e^{\eta \| X \|^{2}} \big]& = \E\big[ \E\big [e^{\sqrt{2}\eta X \cdot Z} \big| X\big] \big] =\E\big[ \E\big[e^{\sqrt{2}\eta X \cdot Z} \big| Z\big]\big] \\
&\le \E\big [ e^{\beta^{2} \eta \| Z \|^{2}}\big] = (1-2\beta^{2} \eta)^{-d/2},\numberthis\label{EQ: subgaussian}
\end{align*}
where the last equality is because $\| Z \|^{2}$ has the $\chi^{2}$-distribution with $d$-degrees of freedom (cf. \cite[Remark 2.3]{hsu2012tail}).

\begin{lemma}[Sufficient condition for \eqref{eq: chi2 condition}]\label{LEMMA:chi_condition}
\label{lem: chi2 condition}
If $P$ is $\beta$-sub-Gaussian for some $\beta < \sigma/\sqrt{2}$, then Condition \eqref{eq: chi2 condition} holds. 
\end{lemma}

The proof also derives an explicit bound on the integral in~\eqref{eq: chi2 condition}.
Lemma \ref{LEMMA:chi_condition} improves upon \cite[Proposition 3]{Goldfeld2019convergence} that shows that $\E\big[\gchi(P_n\|P)\big] =  O(n^{-1})$ for $\beta$-sub-Gaussian $P$ with $\beta < \sigma/2$. Proposition 4 in \cite{Goldfeld2019convergence} shows that if $\beta > \sqrt{2}\sigma$, then $\E\big[\gchi(P_n\|P)\big]$ need not be finite, so in general a sub-Gaussian condition on $P$ is necessary to control $\E\big[\gchi(P_n\|P)\big]$.

Finally, we point out that if $P = \mathcal{N}_{\beta}$, then \eqref{eq: chi2 condition} holds for any $\beta > 0$. More generally, the following holds.

\begin{lemma}\label{lem: gaussian}
Suppose that $P$ is $\beta$-sub-Gaussian for some $\beta > 0$, and $X - \E[X]$, for $X \sim P$, has a Lebesgue density bounded from below by $c e^{-\| x \|^{2}/(2\gamma)}$ for some positive constants $c,\gamma$. Then, Condition \eqref{eq: chi2 condition} holds if $\beta < \sqrt{\frac{\sigma^{2}+2\gamma}{2}}$. In particular, Condition \eqref{eq: chi2 condition} holds if $P$ is Gaussian with covariance matrix $\Sigma$ and $\lambda_{\max} < \lambda_{\min}+\sigma^{2}/2$, where $\lambda_{\max}$ and $\lambda_{\min}$ are the maximum and minimum eigenvalues of $\Sigma$, respectively. 
\end{lemma}



\section{Proofs}


\subsection{Proof of Theorem \ref{thm: CLT TV}}

Our argument relies on the CLT in $L^{1}(\R^{d})$; cf. Appendix \ref{APPEN:CLT in Lp}. 
Recall that $P_{n}*\Gauss$ has density $n^{-1}\sum_{i=1}^{n}Y_{i}(x) =: \overline{Y}_{n}(x)$ with $Y_{i}(x) := \gauss(x-X_{i})$. 
The process $Y_{i}$ is jointly measurable and has paths in $L^{1}(\R^{d})$; indeed,
\[
\| Y_{i} \|_{1} = \int_{\R^{d}} \gauss(x-X_{i})\dd x = 1.
\]
In addition, $\EE[Y_i] = P*\gauss$. By Theorem \ref{thm: CLT Lp}, $\sqrt{n}\left(\overline{Y}_{n} - P*\gauss\right)$ converges weakly in $L^{1}(\R^{d})$ to a Gaussian variable if and only if $\int_{\R^{d}}\mspace{-5mu}\sqrt{\Var (Y_{1}(x))}\dd x < \infty$.

One readily verifies that the limit Gaussian variable is $B_{P}^{(\sigma)}$, and by the continuous mapping theorem, we have \begin{align*}
\sqrt{n} \gtv(P_n,P) &= \frac{1}{2} \int_{\R^{d}} \sqrt{n} |\overline{Y}_{n}(x) - P*\gauss(x)|\dd x \\
&\stackrel{d}{\to} \frac{1}{2} \int_{\R^{d}} \big| B_{P}^{(\sigma)}(x)\big|\dd x.\numberthis
\end{align*}

In addition,
\begin{align*}\E\Big[\gtv(P_n,P)\Big]&=\frac{1}{2} \int_{\R^{d}} \E\big[\big|\overline{Y}_{n}(x) - P*\gauss(x)\big|\big]\dd x \\
&\le \frac{1}{2} \int_{\R^{d}}  \sqrt{\E\big[\big|\overline{Y}_{n}(x) - P*\gauss(x)\big|^2\big]}\dd x\\
&=\frac{1}{2\sqrt{n}} \int_{\R^{d}} \sqrt{\Var_{P}(\varphi (x-\cdot))}\dd x.\numberthis
\end{align*}
This completes the proof. \qed


\subsection{Proof of Lemma \ref{lem: TV condition}}
We first note that for $X \sim P$, 
\begin{equation}
\begin{split}
\Var_{P}\big(\gauss(x-\cdot)\big) 
&\le \E[\varphi_{\sigma}^{2}(x - X)] \\
&=\frac{1}{(2\pi \sigma^{2})^{d}} \int_{\R^{d}} e^{-\|x-y\|^{2}/\sigma^{2}} \dd P(y).
\end{split}
\label{eq: TV moment 0}
\end{equation}
Splitting the integral over $\R^{d}$ into $\|y\| \le \|x\|/2$ and $\|y\| > \|x\|/2$, we have 
\begin{equation}
\begin{split}
&\int_{\R^{d}} e^{-\|x-y\|^{2}/\sigma^{2}} \dd P(y) \\
& \le \int_{\|y\| \le \|x\|/2} e^{-\|x-y\|^{2}/\sigma^{2}} \dd P(y)  + \Prob \big(\|X\| > \|x\|/2\big). 
\end{split}
\label{eq: TV moment 1}
\end{equation}
Changing to polar coordinates, we have 
\begin{equation}
\begin{split}
&\int_{\R^{d}} \sqrt{\Prob \big(\|X\| > \|x\|/2\big)} \dd x \\
&=\frac{2^{d+1}\pi^{d/2}}{\Gamma(d/2)} \int_{0}^{\infty} t^{d-1}\sqrt{\Prob\big(\|X\| > t\big)} \dd t. 
\end{split}
\label{eq: TV moment 2}
\end{equation}
Second, recalling that $\|x-y\|^{2} \ge \|x\|^{2}/2-\|y\|^{2}$,  we have 
\begin{align*}
&\int_{\|y\| \le \|x\|/2} e^{-\|x-y\|^{2}/\sigma^{2}} \dd P(y)\\
&\quad\le e^{-\|x\|^{2}/(4\sigma^{2})} \int_{\|y\| \le \|x\|/2} \dd P(y)\le e^{-\|x\|^{2}/(4\sigma^{2})}. 
\numberthis
\label{eq: TV moment 3}
\end{align*}
The square root of the RHS integrates to $(16 \pi \sigma^{2})^{d/2}$. 
Combining (\ref{eq: TV moment 0})--(\ref{eq: TV moment 3}), we obtain inequality (\ref{eq: TV moment bound}). Finally, if $P$ has finite $(2d+\epsilon)$ moments, then by Markov's inequality, we have 
\[
t^{d-1} \sqrt{\Prob \big(\|X\| > t\big)} \le \min \{ t^{d-1}, \sqrt{\E[\|X\|^{2d+\epsilon}]}/t^{1+\epsilon/2} \}.
\]
The RHS is integrable on $[0,\infty)$. This completes the proof. 
\qed


\subsection{Proof of Proposition \ref{cor: TV condition}}\label{SUBSEC: TV condition proof}

We divide the proof into two steps. 

\textbf{Step 1:} We show that 
if $\sqrt{n}\gtv(P_n,P)$ is stochastically bounded, then its first moment is bounded (in $n$). Let $S_{n} = \sum_{i=1}^{n} (\gauss(x-X_{i}) - P*\gauss(x))$. By Hoffmann-J{\o}rgensen's inequality (see \cite[Proposition 6.8]{LeTa1991}), we have 
\[
\E\big[\| S_{n} \|_{1}\big] \lesssim \E\left [\max_{1 \le i \le n} \big\| \gauss(\cdot-X_{i}) - P*\gauss\big\|_{1}\right] + t_{n,0},
\]
where $t_{n,0} = \inf \big\{ t > 0 : \PP\big(\max_{1 \le m \le n} \| S_{m} \|_{1} > t\big) \le  \frac{1}{8} \big\}$. The first term on the RHS is bounded by $2$. In addition, by Montgomery-Smith's inequality \cite[Corollary 3]{Montgomery-Smith1994}, there exists a universal constant $c$ such that 
\[
t_{n,0} \le \inf \big\{ t > 0 : \PP\big( \| S_{n} \|_{1} > t\big) \le  c \big\}.
\]
If $\gtv(P_n,P) = \| S_{n}/\sqrt{n} \|_{1}$ is stochastically bounded, then $\sup_{n}t_{n,0}/\sqrt{n} < \infty$, which implies $\sup_{n} \E\big[\|S_{n}/\sqrt{n} \|_1\big] < \infty$.

\textbf{Step 2:} We prove that if $\sqrt{n}\E\big[\gtv(P_n,P)\big]$ is bounded, then Condition \eqref{eq: TV condition} holds.  Let $k$ be any positive integer. By Fubini's theorem
\begin{align*}
\sqrt{n}\E\Big[\gtv(P_n,P)\Big] 
&\mspace{-3mu}\ge\mspace{-3mu}\frac{1}{2}\mspace{-3mu} \int_{\R^{d}}\mspace{-3mu} \E\big[\big(\mspace{-3mu}\sqrt{n}\big| \overline{Y}_{\mspace{-4mu}n}\mspace{-1mu}(x)\mspace{-2mu} -\mspace{-2mu} P\mspace{-2mu}*\mspace{-2mu}\gauss(x)\big|\big)\mspace{-3mu}\wedge\mspace{-3mu} k\big]\mspace{-3mu}\dd x.
\end{align*}
Since $|Y_i(x)|\mspace{-3mu} \le\mspace{-3mu} (2\pi \sigma^{2})^{-d/2} =: C_{d,\sigma}$, the central limit theorem implies that
\begin{equation}
 \lim_{n} \E\Big[\big(\sqrt{n}\big| \overline{Y}_{n}(x)- P*\gauss(x)\big|\big) \wedge k\Big] = \E\Big[\big|B_{P}^{(\sigma)}(x)\big| \wedge k\Big] \label{eq: uniformity}
\end{equation}
for any $x \in \R^{d}$. Indeed, let $Q_x$ denote the image measure of $P$ under the map $y \mapsto \gauss(y-x)$, and let $\cQ = \{ Q_{x} : x \in \R^{d} \}$. Each $Q \in \cQ$ is supported in $[-C_{d,\sigma},C_{d,\sigma}]$. 
For each $Q \in \cQ$, let $\xi_{1}^{Q},\dots,\xi_{n}^{Q}$ be i.i.d. with common distribution $Q$. Let $\sigma_{Q}^{2}$ denote the variance of $Q$ (note that $\sigma_{Q_x}^{2} = \Var_{P}(\varphi(x-\cdot)) = \Var \big(B_{P}^{(\sigma)}(x)\big)$). Then,
the central limit theorem implies that 
\[
\sum_{i=1}^{n} \frac{\xi_{i}^{Q} - \E[\xi_{i}^{Q}]}{\sqrt{n}} \stackrel{d}{\to} \cN_{\sigma_{Q}^{2}}. 
\]
Since $y \mapsto |y| \wedge k$ is bounded (by $k$) and (Lipschitz) continuous, the  convergence (\ref{eq: uniformity}) follows from the definition of weak convergence.  
Together with Fatou's lemma, we have
\[
\liminf_{n} \sqrt{n}\E[\delta_{\mathsf{TV}}(P_{n}*\Gauss,P*\Gauss)] \ge
\frac{1}{2} \int_{\R^{d}}\E[|B_{P}^{(\sigma)}(x)| \wedge k]\dd x.
\]
Taking $k \to \infty$, we conclude that 
\[
\begin{split}
&\liminf_n\sqrt{n}\E\Big[\gtv(P_n,P)\Big] \ge \frac{1}{2} \int_{\R^{d}} \E\big[\big|B_{P}^{(\sigma)}(x)\big|\big]\dd x \\
&\quad = \frac{1}{\sqrt{2\pi}}\int_{\R^{d}} \sqrt{\Var_{P}(\gauss(x-\cdot))}\dd x,
\end{split}
\]
where the second equality is because $\E[|\xi|] = \sqrt{2/\pi} \sqrt{\E[\xi^{2}]}$ for a centered Gaussian variable $\xi$. This completes the proof. \qed


\subsection{Proof of Theorem \ref{thm: CLT chi2}}

We will apply the CLT in $L^{2}(P*\Gauss)$; cf. Appendix \ref{APPEN:CLT in Lp}.  Let $Z_{i} := \frac{\gauss(\cdot-X_{i})}{P*\gauss} - 1$.
The process $Z_i$ is jointly measurable and has paths almost surely in $L^{2}(P*\Gauss)$; indeed, 
\[
\E\mspace{-3mu}\left[\int_{\R^{d}}\mspace{-5mu} |Z_{i}(x)|^{2} \dd P*\Gauss(x) \right]\mspace{-3mu} =\mspace{-3mu} \int_{\R^{d}}\mspace{-3mu} \frac{\Var_{P}(\gauss(x-\cdot))}{P*\gauss(x)}\dd x < \infty.
\]
Hence, we apply Theorem \ref{thm: CLT Lp} to conclude that $\sum_{i=1}^{n} Z_{i}/\sqrt{n}$ converges weakly in $L^{2}(P*\Gauss)$ to a Gaussian variable. The limit Gaussian variable is $B_{P}^{(\sigma)}/P*\gauss$, and by the continuous mapping theorem, we have 
\begin{align*}
&n \gchi(P_n\|P) 
= \int_{\R^{d}} \left| \sum_{i=1}^{n}Z_{i}(x)/\sqrt{n} \right |^{2}\dd P*\Gauss (x) \\
& \stackrel{d}{\to} \int_{\R^{d}} \left | \frac{B_{P}^{(\sigma)}(x)}{P*\gauss(x)} \right |^{2}\dd P*\Gauss(x) = \int_{\R^{d}}  \frac{\big|B_{P}^{(\sigma)}(x)\big|^2}{P*\gauss(x)}\dd x.\numberthis
\end{align*}

Finally, since $L^{2}(P*\Gauss)$ is a Hilbert space, we have
\begin{align*}
&n \E\big[\gchi(P_n\|P)\big] = \E \left [ \bigg \| \sum_{i=1}^{n}Z_{i}/\sqrt{n}  \bigg \|_{L^{2}(P*\Gauss)}^{2} \right ] \\
& =\sum_{i=1}^{n} \frac{\E\big[\| Z_{i} \|^{2}_{L^{2}(P*\Gauss)}\big]}{n}=\int_{\R^{d}}\frac{\Var_{P}(\gauss(x-\cdot))}{P*\gauss(x)}\dd x.\numberthis
\end{align*}
This completes the proof.\qed


\subsection{Proof of Lemma \ref{lem: chi2 condition}}
We first note that
\[
\begin{split}
&\int_{\R^{d}} \frac{\Var_{P}(\gauss(x-\cdot))}{P*\gauss(x)}\dd x \le \int_{\R^{d}} \int_{\R^{d}} \frac{\gauss(x-y)^{2}}{P*\varphi_{ \sigma}(x)}\dd x \dd P(y) \\
&= \int_{\R^{d}} e^{-\| y \|^{2}/\sigma^{2}} \frac{1}{(2\pi \sigma^{2})^{d}} \int_{\R^{d}} \frac{e^{2y \cdot x/\sigma^{2}-\|x\|^{2}/\sigma^{2}}}{P*\varphi_{ \sigma}(x)}\dd x \dd P(y).
\end{split}
\]
Since $\|x\mspace{-3mu}-\mspace{-3mu}y\|^{2}\mspace{-3mu} \le\mspace{-3mu} (1\mspace{-3mu}+\mspace{-3mu}\eta)\|x\|^{2}\mspace{-3mu}+\mspace{-3mu}(1\mspace{-3mu}+\mspace{-3mu}\frac{1}{\eta})\|y\|^{2},\ \mspace{-3mu}\forall \eta\mspace{-3mu} \in\mspace{-3mu} (0,1)$, we have 
\[
P*\gauss(x) \ge \frac{e^{-(1+\eta)\|x\|^{2}/(2\sigma^{2})}}{(2\pi \sigma^{2})^{d/2}}  \underbrace{\int_{\R^{d}} e^{-(1+\eta^{-1})\|y\|^{2}/(2\sigma^{2})} \dd P(y)}_{=: C_{P,\eta}},
\]
so that 
\[
\begin{split}
\frac{1}{(2\pi \sigma^{2})^{d}} \int_{\R^{d}} &\frac{e^{2y \cdot x/\sigma^{2}-\|x\|^{2}/\sigma^2}}{P*\varphi_{ \sigma}(x)}\dd x\\
&\le \frac{1}{(1-\eta)^{d/2} C_{P,\eta}} \int_{\R^{d}} e^{2y \cdot x/\sigma^{2}} \varphi_{\sigma/\sqrt{1-\eta}}(x) \dd x \\
&=\frac{1}{(1-\eta)^{d/2} C_{P,\eta}} \exp   \left ( \frac{2 \| y \|^{2}}{(1-\eta)
\sigma^{2}} \right). 
\end{split}
\]
Conclude that 
\[
\begin{split}
\int_{\R^{d}}& \frac{\Var_{P}(\gauss(x-\cdot))}{P*\gauss(x)}\dd x\\
&\le \frac{1}{(1-\eta)^{d/2} C_{P,\eta}} \int_{\R^{d}} \exp \left ( \frac{(1+\eta)\| y \|^{2}}{(1-\eta)\sigma^{2}} \right ) \dd P(y) \\
&= \frac{1}{(1-\eta)^{d/2} C_{P,\eta}} \left [ 1-\frac{2(1+\eta)\beta^{2}}{(1-\eta)\sigma^{2}} \right]^{-d/2}
\end{split}
\]
by (\ref{EQ: subgaussian}), 
provided that
$\frac{(1+\eta)}{(1-\eta)\sigma^{2}} < \frac{1}{2\beta^{2}}$, i.e., $\beta < \sigma \sqrt{\frac{1-\eta}{2(1+\eta)}}$.
Since $\eta \in (0,1)$ is arbitrary, the desired result follows. \qed


\subsection{Proof of Lemma \ref{lem: gaussian}}
By translation invariance of Lebesgue measure, we may~assume $P$ has zero mean. The proof is similar to that of Lemma \ref{lem: chi2 condition}, so we only outline required modifications. Observe that 
\[
\begin{split}
P*\gauss (x) &\gtrsim e^{-\|x\|^{2}/(2\sigma^{2})} \int_{\R^{d}} e^{x \cdot y/\sigma^{2}} \varphi_{\sqrt{1/(\sigma^{-2}+\gamma^{-1})}} (y) \dd y \\
&=e^{-\|x\|^{2}/2(\sigma^{2} + \gamma)}, 
\end{split}
\]
so that 
\[
\begin{split}
\int_{\R^{d}} \frac{e^{\frac{2y \cdot x-\|x\|^{2}}{\sigma^2}}}{P*\varphi_{ \sigma}(x)}\dd x &\lesssim \mspace{-3mu}\int_{\R^{d}} \mspace{-5mu}\exp \mspace{-1.5mu}\left ( \frac{2 x \cdot y}{\sigma^{2}}\mspace{-2mu} -\mspace{-2mu} \frac{\sigma^{2}+2\gamma}{2\sigma^{2} (\sigma^{2}+\gamma)} \| x \|^{2}\mspace{-3mu} \right)\mspace{-3mu} \dd x \\\mspace{-3mu}
&\lesssim \exp \left ( \frac{2(\sigma^{2}+\gamma)\| y \|^{2}}{\sigma^{2}(\sigma^{2}+2\gamma)}  \right ).
\end{split}
\]
Multiplying $e^{-\|y\|^{2}/\sigma^{2}}$ to the RHS leads to $\exp \left ( \frac{\|y\|^{2}}{\sigma^{2}+2\gamma} \right )$, whose integration w.r.t. $P$ is finite as soon as $\frac{1}{\sigma^{2} + 2\gamma} < \frac{1}{2\beta^{2}}$, i.e., $\beta < \sqrt{\frac{\sigma^{2} + 2\gamma}{2}}$.
If $P$ is Gaussian with covariance matrix $\Sigma$, then we may take $\beta = \sqrt{\lambda_{\max}}$ and $\gamma = \lambda_{\min}$.
\qed

\appendices

\section{Central Limit Theorem in $L^{p}$ Spaces}
\label{APPEN:CLT in Lp}

We summarize basic CLT results in $L^{p}$ spaces with $1 \le p < \infty$. Let $\mu$ be a $\sigma$-finite measure on a measurable space $(S,\mathcal{S})$ with $\mathcal{S}$ being countably generated, and let $L^{p} = L^{p}(S,\mathcal{S},\mu),  1 \leq p < \infty$ be the space of real-valued measurable functions $f$ on $S$ such that $\| f \|_{p} := ( \int_{S} |f(s)|^{p}\dd \mu (s) )^{1/p} < \infty$. 
As usual, we identify functions $f,g$ on $S$ if $f = g$ $\mu$-a.e.; under this identification, the space $(L^{p}, \| \cdot \|_{p})$ is a separable Banach space (and thus Polish). 
Recall that any Borel measurable random variable with values in a Polish space is tight (Radon) by Ulam's theorem \cite[Theorem 7.1.3]{Du2002}. A Borel measurable random variable with values in $L^{p}$ is called an $L^{p}$-valued random variable.

For $1 \leq p < \infty$, let $q$ be its  conjugate index, i.e., $\frac{1}{p} +\frac{1}{q} = 1$ ($q=\infty$ if $p=1$). 
For an $L^{p}$-valued random variable $X$, let
\[
\tilde{X} (f) = \int_{S} f(s) X(s)\dd \mu (s), \ f \in L^{q},
\]
which is a stochastic process index by $L^{q}$, the dual space of $L^{p}$. An $L^{p}$-valued random variable $G$ is Gaussian if $\{ \tilde{G}(f) :  f \in L^{q}\}$ is a Gaussian process. 

Let $X$ be an $L^{p}$-valued random variable  such that $\E[\tilde{X}(f)] = 0$ and $\E[\tilde{X}(f)^{2}] < \infty$ for all $f \in L^{q}$. $X$ is said to be pre-Gaussian if there exists a centered $L^{p}$-valued Gaussian random variable $G$ with the same covariance function as $X$, i.e, $\E[\tilde{G}(f)\tilde{G}(g)]=\E[\tilde{X}(f)\tilde{X}(g)]$ for all $f,g \in L^{q}$.

Note that if $X$ is a jointly measurable (Gaussian) process with paths in $L^{p}$ then $X$ can be identified to be an $L^{p}$-valued (Gaussian, resp.) random variable and vice versa; cf. \cite{Bycz1977}.

\begin{theorem}[Proposition 2.1.11 in \cite{VaWe1996}]
\label{thm: CLT Lp}
Let $1 \le p < \infty$,~and $X,X_{1},\dots,X_{n}$ be i.i.d. $L^{p}$-valued random variables with zero mean (in the sense of Bochner). The following are~equivalent:
\begin{enumerate}
\item[(i)] There exists a centered Gaussian variable $G$ in $L^{p}$ with the same covariance function as $X$ such that $S_{n}:=\sum_{i=1}^{n}X_{i}/\sqrt{n}$ converges weakly in $L^{p}$ to $G$.
\item[(ii)] $\int_{S} (\E[ |X(s)|^{2} ])^{p/2} \dd \mu (s) < \infty$ and $\PP\big (\| X \|_{p} > t\big) = o(t^{-2})$ as $t \to \infty$.
\end{enumerate} 
\end{theorem}

\begin{proof}
This is \cite[Proposition 2.1.11]{VaWe1996} but it does not contain a proof. The proposition cites Theorem 10.10 in \cite{LeTa1991}, but we shall complement some arguments for the reader's convenience. Theorem 10.10 in \cite{LeTa1991} and the discussion following it imply that $S_{n}$ converges weakly in $L^{p}$ if and only if $X$ is pre-Gaussian and $\PP\big(\| X \|_{p} >t\big) = o(t^{-2})$. So we only have to verify that $X$ is pre-Gaussian if and only if $\int_{S} (\E[ |X(s)|^{2} ])^{p/2}\dd \mu (s) < \infty$. 

If $X$ is pre-Gaussian, then 
\[
\infty\mspace{-3mu} >\mspace{-2mu} \E[\| G \|_{p}^{p}]\mspace{-3mu} =\mspace{-5mu} \int_{S} \mspace{-4mu}\E[|G(\mspace{-0.5mu}s\mspace{-0.5mu})|^{p}]\dd \mu(\mspace{-0.5mu}s\mspace{-0.5mu})\mspace{-3mu} =\mspace{-2mu} c_{p}\mspace{-4mu} \int_{S}\mspace{-5mu} \big(\E[|X(\mspace{-0.5mu}s\mspace{-0.5mu})|^{2}]\big)^{\mspace{-2mu}p/2}\mspace{-3mu}\dd \mu(\mspace{-0.5mu}s\mspace{-0.5mu})
\]
with $c_{p} = \E[|Z|^{p}]$ for $Z \sim \cN (0,1)$. 
For the ``if'' part, let $B=L^{p}$ and $B' = L^{q}$.  Since $B$ is separable, we may assume that $X$ is defined on a probability space $(\Omega,\mathcal{F},\PP)$ with $\mathcal{F}$ countably generated. Then $L^{2}(\Omega, \mathcal{F},\PP)$ is separable and so there exists a complete orthonormal system $(h_{i})_{i=1}^{\infty}$ of $L^{2}(\Omega,\mathcal{F},\PP)$. Define a Gaussian variable $G_{n}(\cdot)= \sum_{i=1}^{n} Z_{i}\E[h_{i}X(\cdot)]$ (defined as a $B$-valued random variable), where $(Z_{i})_{i=1}^{\infty}$ is a sequence of independent $\cN (0,1)$ random variables.
We show that $G_{n}$ converges in $L^{p}(B) = L^{p}(\Omega,\mathcal{F},\PP; B)$. To this end, observe that for $n > m$, 
\[
\begin{split}
&\E\big[ \| G_{n} - G_{m} \|_{p}^{p} \big]  = \int_{S} \E \left [ \left | \sum_{i=m+1}^{n} Z_{i} \E[ h_{i} X (s)] \right |^{p} \right ]\dd \mu (s) \\
&\quad =c_{p} \int_{S} \left \{  \sum_{i=m+1}^{n} \{  \E[h_{i}X(s)] \}^{2} \right \}^{p/2}\dd \mu (s).
\end{split}
\]
Hence, $G_{n}$ is Cauchy in $L^{p}(B)$ and so there exists a $B$-valued random variable $G$ such that $\E\big[\| G_{n} - G \|_{p}^{p}\big] \to 0$. It is not difficult to see that $G$ is Gaussian with $\E[\tilde{G}(f)] = \lim_{n}\E[\tilde{G}_{n}(f)] = 0$ and $\E[\tilde{G}(f)^{2}] = \lim_{n} \E[\tilde{G}_{n}(f)^{2}] = \E[\tilde{X}(f)^{2}]$ for every $f \in B'$. Thus, $X$ is pre-Gaussian. 
\end{proof}

\newpage

\bibliographystyle{IEEEtran}
\bibliography{ref}

\begin{thebibliography}{10}
\providecommand{\url}[1]{#1}
\csname url@samestyle\endcsname
\providecommand{\newblock}{\relax}
\providecommand{\bibinfo}[2]{#2}
\providecommand{\BIBentrySTDinterwordspacing}{\spaceskip=0pt\relax}
\providecommand{\BIBentryALTinterwordstretchfactor}{4}
\providecommand{\BIBentryALTinterwordspacing}{\spaceskip=\fontdimen2\font plus
\BIBentryALTinterwordstretchfactor\fontdimen3\font minus
  \fontdimen4\font\relax}
\providecommand{\BIBforeignlanguage}[2]{{%
\expandafter\ifx\csname l@#1\endcsname\relax
\typeout{** WARNING: IEEEtran.bst: No hyphenation pattern has been}%
\typeout{** loaded for the language `#1'. Using the pattern for}%
\typeout{** the default language instead.}%
\else
\language=\csname l@#1\endcsname
\fi
#2}}
\providecommand{\BIBdecl}{\relax}
\BIBdecl

\bibitem{csiszar2004information}
I.~Csisz{\'a}r and P.~C. Shields, ``Information theory and statistics: A
  tutorial,'' \emph{Foundations and Trends{\textregistered} in Communications
  and Information Theory}, vol.~1, no.~4, pp. 417--528, Dec. 2004.

\bibitem{nowozin2016f}
S.~Nowozin, B.~Cseke, and R.~Tomioka, ``f-gan: Training generative neural
  samplers using variational divergence minimization,'' in \emph{Advances in
  Neural Information Processing Systems (NeurIPS-2016)}, Barcelona, Spain, Dec.
  2016, pp. 271--279.

\bibitem{uehara2016generative}
M.~Uehara, I.~Sato, M.~Suzuki, K.~Nakayama, and Y.~Matsuo, ``Generative
  adversarial nets from a density ratio estimation perspective,'' \emph{arXiv
  preprint arXiv:1610.02920}, 2016.

\bibitem{devroye1990no}
L.~Devroye and L.~Gyorfi, ``No empirical probability measure can converge in
  the total variation sense for all distributions,'' \emph{Ann. Stat.},
  vol.~18, no.~3, pp. 1496--1499, 1990.

\bibitem{Goldfeld2019convergence}
Z.~Goldfeld, K.~Greenewald, Y.~Polyanskiy, and J.~Weed, ``Convergence of
  smoothed empirical measures with applications to entropy estimation,''
  \emph{arXiv preprint arXiv:1905.13576}, May 2019.

\bibitem{Goldfeld2020GOT}
Z.~Goldfeld and K.~Greenewald, ``Gaussian-smoothed optimal transport: Metric
  structure and statistical efficiency,'' in \emph{International Conference on
  Artificial Intelligence and Statistics (AISTATS-2020)}, Palermo, Sicily,
  Italy, Jun. 2020.

\bibitem{goldfeld2020limit}
Z.~Goldfeld and K.~Kato, ``Limit distribution theory for smooth {Wasserstein}
  distance with applications to generative modeling,'' \emph{arXiv preprint
  arXiv:2002.01012}, Feb. 2020.

\bibitem{LeTa1991}
M.~Ledoux and M.~Talagrand, \emph{Probability in {B}anach {S}paces:
  {I}soperimetry and {P}rocesses}.\hskip 1em plus 0.5em minus 0.4em\relax
  Springer. New York, 1991.

\bibitem{VaWe1996}
A.~van~der Vaart and J.~A. Wellner, \emph{Weak Convergence and Empirical
  Processes: With Applications to Statistics}.\hskip 1em plus 0.5em minus
  0.4em\relax Springer, 1996.

\bibitem{GiNi2016}
E.~Gin\'e and R.~Nickl, \emph{Mathematical {F}oundations of
  {I}nfinite-{D}imensional {S}tatistical {M}odels}.\hskip 1em plus 0.5em minus
  0.4em\relax Cambridge University Press, 2016.

\bibitem{del1999central}
E.~del Barrio, E.~Gin{\'e}, and C.~Matr{\'a}n, ``Central limit theorems for the
  {Wasserstein} distance between the empirical and the true distributions,''
  \emph{The Annals of Probability}, vol.~27, no.~2, pp. 1009--1071, Apr. 1999.

\bibitem{bobkov2014berry}
S.~G. Bobkov, G.~P. Chistyakov, and F.~G{\"o}tze, ``Berry--{E}sseen bounds in
  the entropic central limit theorem,'' \emph{Probab. Theory Relat. Fields},
  vol. 159, no. 3-4, pp. 435--478, 2014.

\bibitem{bobkov2019renyi}
------, ``R{\'e}nyi divergence and the central limit theorem,'' \emph{Ann.
  Probab.}, vol.~47, no.~1, pp. 270--323, 2019.

\bibitem{courtade2020bounds}
T.~A. Courtade, ``Bounds on the {P}oincar{\'e} constant for convolution
  measures,'' in \emph{Annales de l'Institut Henri Poincar{\'e},
  Probabilit{\'e}s et Statistiques}, vol.~56, no.~1.\hskip 1em plus 0.5em minus
  0.4em\relax Institut Henri Poincar{\'e}, 2020, pp. 566--579.

\bibitem{bobkov2019entropic}
S.~G. Bobkov and A.~Marsiglietti, ``Entropic {CLT} for smoothed convolutions
  and associated entropy bounds,'' \emph{arXiv preprint arXiv:1903.03666},
  2019.

\bibitem{FournierGuillin2015}
N.~Fournier and A.~Guillin, ``On the rate of convergence in wasserstein
  distance of the empirical measure,'' \emph{Probability Theory and Related
  Fields}, vol. 162, pp. 707--738, 2015.

\bibitem{nguyen2010estimating}
X.~Nguyen, M.~J. Wainwright, and M.~I. Jordan, ``Estimating divergence
  functionals and the likelihood ratio by convex risk minimization,''
  \emph{IEEE Trans. Inf. Theory}, vol.~56, no.~11, pp. 5847--5861, Nov. 2010.

\bibitem{sriperumbudur2009integral}
B.~K. Sriperumbudur, K.~Fukumizu, A.~Gretton, B.~Sch{\"o}lkopf, and G.~R.
  Lanckriet, ``On integral probability metrics, $\phi$-divergences and binary
  classification,'' \emph{arXiv preprint arXiv:0901.2698}, 2009.

\bibitem{villani2008optimal}
C.~Villani, \emph{Optimal {T}ransport: Old and {N}ew}.\hskip 1em plus 0.5em
  minus 0.4em\relax Springer Science \& Business Media, 2008, vol. 338.

\bibitem{muller1997integral}
A.~M{\"u}ller, ``Integral probability metrics and their generating classes of
  functions,'' \emph{Advances in Applied Probability}, vol.~29, no.~2, pp.
  429--443, 1997.

\bibitem{Folland1999}
G.~B. Folland, \emph{Real Analysis, Second Edition}.\hskip 1em plus 0.5em minus
  0.4em\relax Wiley, 1999.

\bibitem{Mc1989}
C.~McDiarmid, ``On the method of bounded differences,'' in \emph{Surveys in
  Combinatorics}.\hskip 1em plus 0.5em minus 0.4em\relax Cambridge University
  Press, 1989.

\bibitem{hsu2012tail}
D.~Hsu, S.~Kakade, and T.~Zhang, ``A tail inequality for quadratic forms of
  subgaussian random vectors,'' \emph{Electronic Communications in
  Probability}, vol.~17, 2012.

\bibitem{Montgomery-Smith1994}
S.~J. Montgomery-Smith, ``Comparison of sums of independent identically
  distributed random variables,'' \emph{Probab. Math. Statist.}, vol.~14, pp.
  281--285, 1994.

\bibitem{Du2002}
R.~M. Dudley, \emph{Real Analysis and Probability}.\hskip 1em plus 0.5em minus
  0.4em\relax Cambridge University Press, 2002.

\bibitem{Bycz1977}
T.~Byczkowski, ``Gaussian measures on {$L_{p}$} spaces $0 \le p <\infty$,''
  \emph{Studia Math.}, vol.~59, pp. 249--261, 1977.

\end{thebibliography}

\end{document}